\documentclass{amsart}
\usepackage{amsmath}
\usepackage{amscd}
\usepackage{amssymb}

 \newtheorem{thm}{Theorem}
 \newtheorem{prop}{Proposition}
 
 \newtheorem{lem}{Lemma}
 \newtheorem{cl}{Claim}
 \newtheorem{conj}{Conjecture}
 \theoremstyle{definition}
 \newtheorem{prob}{Problem}

\newcommand{\Gama}{\varGamma}
\newcommand{\eps}{\varepsilon}
\newcommand{\imp}{\,\Rightarrow\,}
\newcommand{\bul}{$\bullet$}
\newcommand{\sbs}{\subseteq}
\newcommand{\ssm}{\smallsetminus}
\newcommand{\id}{\operatorname{id}}
\newcommand{\lra}{\leftrightarrow}
\newcommand{\co}{\colon}
\newcommand{\la}{\langle}
\newcommand{\ra}{\rangle}
\newcommand{\St}{\mathcal{S}}
\newcommand{\ap}{{\mathbf{a}}}
\newcommand{\bp}{{\mathbf{b}}}
\newcommand{\X}{X^+}
\newcommand{\Y}{Y^+}
\newcommand{\Z}{Z^+}

\newenvironment{enum}{\parindent0pt%
\begin{list}{}{%
\setlength{\itemindent}{0ex}
\setlength{\labelwidth}{15pt}
\setlength{\labelsep}{6pt}
\setlength{\leftmargin}{21pt}
\setlength{\listparindent}{0pt}
\setlength{\itemsep}{0ex}
\setlength{\topsep}{0ex}
\setlength{\parsep}{0.2em} 
}
}{\end{list}}

\begin{document}

\title[Local Embeddability of Groups into Finite IP Loops]%
{The Finite Embedding Property for IP Loops \\
and Local Embeddability of Groups \\
into Finite IP~Loops}

\author[M.~Vodi\v{c}ka and P.~Zlato\v{s}]%
{Martin Vodi\v{c}ka and Pavol Zlato\v{s}}

\address{%
{\sl Martin Vodi\v{c}ka}
\newline\indent
{\sl Max-Planck-Institut f\"ur Mathematik in den Naturwissenschaften,}
\newline\indent
{\sl Inselstrasse 22, 04103 Leipzig, Germany}
\newline\indent
{\tt vodicka@mis.mpg.de}}

\address{%
{\sl Pavol Zlato\v{s}}
\newline\indent
{\sl Faculty of Mathematics, Physics and Informatics,
Comenius University,}
\newline\indent
{\sl Mlynsk\'a dolina, 842\,48~Bratislava, Slovakia}
\newline\indent
{\tt zlatos@fmph.uniba.sk}}



\begin{abstract}
We prove that the class of all loops with the inverse property (IP~loops)
has the Finite Embedding Property (FEP). As a consequence, every group is
locally embeddable into finite IP loops.
\end{abstract}

\maketitle


\noindent
The \textit{Finite Embedding Property} (briefly FEP), was introduced by Henkin
\cite{Hen2} for general algebraic systems already in 1956. For groupoids
(i.e., algebraic structures $(G,\cdot)$ with a single binary operation), which
is sufficient for our purpose, it reads as follows: A class $\mathbf{K}$ of
groupoids has the FEP if for every algebra $(G,\cdot) \in \mathbf{K}$ and each
nonempty subset $X \sbs G$ there is a \textit{finite} algebra
$(H, *) \in \mathbf{K}$ extending $(X,\cdot)$, i.e., $X \sbs H$ and
$x \cdot y = x * y$ for all $x,y \in X$, such that $x \cdot y \in X$. Using this
notion an earlier result of Henkin \cite{Hen1} can be stated as follows: The class
of all abelian groups has the FEP (see also Gr\"atzer \cite{Grt}).

A more general notion of local embeddability can be traced back to even earlier
papers by Mal\!'tsev \cite{Mal1}, \cite{Mal2} (see also the posthumous monograph
\cite{Mal3}). It was explicitly (re)introduced and studied in detail mainly for
groups by Vershik and Gordon \cite{VG}: A groupoid $(G,\cdot)$ is \textit{locally
embeddable into a class of gropupoids $\mathbf{M}$} if for every $X \sbs G$
there is an $(H,*) \in \mathbf{M}$ such that $X \sbs H$ and $x \cdot y = x * y$
for all $x,y \in X$ satisfying $x \cdot y \in X$. Informally this means that every
finite cut-out from the multiplication table of $(G,\cdot)$ can be embedded into
an algebra from $\mathbf{M}$. A standard model-theoretic argument shows that this
condition is equivalent to the embeddability of $(G,\cdot)$ into an
\textit{ultraproduct} of algebras from $\mathbf{M}$ (for the ultraproduct
construction see, e.g., Chang, Keisler \cite{CK}).

Thus a class $\mathbf{K}$ has the FEP if and only if every $(G,\cdot) \in \mathbf{K}$
is locally embeddable into the class $\mathbf{K}_{\mathrm{fin}}$ of all finite members
in $\mathbf{K}$. The groups locally embeddable into (the class of all) finite groups
were called \textit{LEF groups} in \cite{VG}. The authors also noticed that, unlike the
abelian ones, not all groups are LEF, in other words, the class of all groups doesn't
have the FEP. This immediately raises the question of finding some classes of finite
grupoids into which all the groups were locally embeddable and which, at the same time,
would be ``as close to groups as possible''. The question is of interest for various
reasons: The class of all LEF groups properly extends the class of all locally residually
finite groups and plays an important role, in dynamical systems, cellular automata, etc.
(see, e.g., Ceccherini-Silberstein, Coornaert \cite{CSC}, Vershik, Gordon \cite{VG}).

Glebsky and Gordon \cite{GG1} have shown that a group is locally embeddable into finite
semigroups if and only if it is an LEF group. It follows that looking for a class of
finite groupoids into which one could locally embed all the groups one has to sacrifice
the associativity condition. They also noticed that the results about extendability of
partial latin squares to (complete) latin squares imply that every group is locally
embeddable into finite quasigroups. Refining slightly the original argument they have
shown that every group can even be locally embedded into finite loops (see also their
survey article \cite{GG2}).

A further decisive step in this direction was done by Ziman \cite{Zim}. Refining
considerably the methods of extension of partial latin squares and preserving some
symmetry conditions (see Lindner \cite{Lind}, Cruse \cite{Crs}), he has shown that
the class of all loops with \textit{antiautomorphic inverses}, i.e., loops with
two-sided inverses satisfying the identity
$$
(xy)^{-1} = y^{-1} x^{-1}
$$
(briefly \textit{AAI~loops}), has the FEP (though he didn't use this notion explicitly).
As a consequence, every group is locally embeddable into finite AAI~loops.

Quasigroups and loops experts consider the class of all AAI~loops still as a ``rather far
going extension'' of the class of all groups. On the other hand, they find the class of
all loops with the \textit{inverse property}, i.e., loops with two-sided inverses
satisfying the identities
$$
x^{-1}(xy) = y = (yx)x^{-1}
$$
(briefly IP~loops), which is a proper subclass of the class of all AAI~loops, a much
more moderate extension of the class of all groups (Dr\'apal \cite{Drap}). In the present
paper we are going to show that Ziman's result can indeed be strengthened in this sense.
Using mainly graph-theoretical methods and Steiner triple systems, we will prove that the
class of all IP~loops still has the FEP. As a consequence, every group is locally
embeddable into finite IP~loops.

For basic definitions and facts about quasigroups and loops the reader is referred
to the monographs Belousov and Belyavskaya \cite{BB} and Pflugfelder \cite{Pflug}.

\section{Formulation of the main results and plan of the proof}\label{1}

\noindent
Our main results are the following two theorems, the second one of which is obviously
an immediate consequence of the first one.

\begin{thm}\label{thm1}
The class of all IP~loops has the Finite Embedding Property.
\end{thm}

\begin{thm}\label{thm2}
Every group can be locally embedded into the class of all IP~loops.
Equivalently, every group can be embedded into some ultraproduct of
finite IP~loops.
\end{thm}

We divide the proof of Theorem~\ref{thm1} into three steps consisting of the three
propositions below. Their formulation requires some notions an notation.

A \textit{partial IP~loop} $(P,\cdot)$ is a set $P$ endowed with a partial binary
operation $\cdot$ defined on a subset $D(P) \sbs P \times P$, called the
\textit{domain} of the operation $\cdot$, satisfying the following three
conditions:
\smallskip

\begin{enum}
\item[(1)]
there is an element $1 \in P$, called the \textit{unit} of $P$, such that
$(1,x), (x,1) \in D(P)$ and $1x = x1 = x$ for all $x \in P$;
\item[(2)]
for each $x \in P$ there is a unique $y \in P$, called the \textit{inverse}
of $x$ and denoted by $y = x^{-1}$, such that $(x,y), (y,x) \in D(P)$ and
$xy = yx = 1$;
\item[(3)]
for any $x,y \in P$ such that $(x,y) \in D(P)$ we have
$\bigl(x^{-1},xy\bigr), \bigl(xy,y^{-1}\bigr) \in D(P)$ and $x^{-1}(xy) = y$,
$(xy)y^{-1} = x$.
\end{enum}
\smallskip

In most cases we will denote a partial IP~loop $(P,\cdot)$ as $P$, only;
the more unambiguous notation $(P,\cdot)$ will be used just in case we need
to distinguish the operations on two or more more (partial) IP~loops.

Obviously, every subset $P \sbs L$ of an IP~loop $L$, such that $1 \in P$
and $x^{-1} \in P$ for each $x \in P$, gives rise to a partial IP~loop with
domain
$$
D(P) = \{(x,y) \in P \times P\co xy \in P\}.
$$

A partial IP~loop $(Q,*)$ is called an \textit{extension} of a partial IP~loop
$(P,\cdot)$ if $P \sbs Q$, $D(P) \sbs D(Q)$ and $x \cdot y = x * y$
for each pair $(x,y) \in D(P)$. Suppressing the signs of the operations, we write
$P \le Q$ or $Q \ge P$. Obviously, the relation $\le$ between partial IP loops
is reflexive, antisymmetric and transitive.

In the absence of associativity there is no obvious way how to define the order
of an element. Nonetheless, the sets of elements of order 2 and 3, respectively,
can still be defined for any partial IP~loop $P$:
\begin{align*}
O_2(P) &= \{x \in P\co (x,x) \in D(P),\ x \ne 1 \ \mathrm{and} \ xx = 1\},\\
O_3(P) &= \{x \in P\co (x,xx) \in D(P),\ x \ne 1 \ \mathrm{and} \ x(xx) = 1\}.
\end{align*}
In other words, for an $x \ne 1$ in $P$ we have $x \in O_2(P)$ if and only if
$x^{-1} = x$, and $x \in O_3(P)$ if and only if $x^{-1} = xx$. The number of
elements of the sets $O_2(P)$, $O_3(P)$ in a finite partial IP~loop $P$ will
be denoted by $o_2(P)$, $o_3(P)$, respectively. In general, the number of
elements of a finite set $A$ is denoted by $\#A$.

\begin{prop}\label{prop1}
Let $(P,\cdot)$ be a finite partial IP~loop. Then there exists a finite partial
IP~loop $(Q,*)$ such that $P \le Q$ and \,$3\,|\,o_3(Q)$.
\end{prop}

A pair $(x,y)$ in a partial IP~loop $P$ will be called a \textit{gap} if
$(x,y) \notin D(P)$. The set of all gaps in $P$ will be denoted by
$$
\Gama(P) = (P \times P) \ssm D(P)
          = \{(x,y) \in P \times P\co (x,y) \notin D(P)\}
$$
Obviously, both $D(P)$, $\Gama(P)$ are binary relation on the set $P$, and a
partial IP~loop $P$ is an IP~loop if and only if it contains no gaps, i.e.,
$\Gama(P) = \emptyset$.

\begin{prop}\label{prop2}
Let $P$ be a finite partial IP~loop such that \,$3\,|\,o_3(P)$. Then there
exists a finite partial IP~loop $Q$ satisfying the following four conditions:
\smallskip

\begin{enum}
\item[\rm{(4)}]\quad
$3\,|\,o_3(Q)$, \ $\#Q \ge 10$, \ $\#Q \equiv 4 \!\pmod 6$ \ and\,
\ $\Gama(Q) \sbs O_2(Q) \times O_2(Q)$.
\end{enum}
\end{prop}

\begin{prop}\label{prop3}
Let $P$ be a finite partial IP~loop satisfying the above conditions {\rm (4)},
such that $\Gama(P) \ne \emptyset$. Then there is a finite partial
IP~loop $Q \ge P$ satisfying the conditions {\rm (4)}, as well, such
that $\#\Gama(Q) < \#\Gama(P)$.
\end{prop}

Theorem \ref{thm1} follows from Propositions~\ref{prop1}, \ref{prop2} and
\ref{prop3}. Indeed, given an IP~loop $L$ and a finite set $X \sbs L$
(which is not a sub-IP~loop of $L$) we can form the finite partial IP~loop
$$
P = X \cup \{1\} \cup X^{-1},
$$
where $X^{-1} = \{x^{-1}\co x \in X\}$, satisfying $X \sbs P \le L$. Then,
using Proposition~\ref{prop1}, we can find a finite partial IP~loop $Q \ge P$
such that $3\,|\,o_3(Q)$. If $\Gama(Q) = \emptyset$ then $Q$ is already a finite
IP~loop extending $P$, and we are done. Otherwise, applying Proposition~\ref{prop2},
we obtain a finite partial IP~loop $Q_1 \ge Q$ satisfying conditions (4) from
Proposition~\ref{prop2}. If $\Gama(Q_1) = \emptyset$ then we are done, again.
Otherwise, we can apply Proposition~\ref{prop3} and get a finite partial
IP~loop $Q_2 \ge Q_1$ satisfying conditions (4), as well, such that
$\#\Gama(Q_2) < \#\Gama(Q_1)$. Iterating this step finitely many times we
finally arrive at some finite partial IP~loop $Q_n$ extending $P$ such that
$\Gama(Q_n) = \emptyset$. Then $Q_n \ge P$ is a finite IP~loop we have been
looking for.

Thus it is enough to prove the Propositions~\ref{prop1}, \ref{prop2} and
\ref{prop3}. This will take place in the next four sections.

\section{Some preliminary results}\label{2}

\noindent
In this section we list the auxiliary results we will use in the proofs of
Propositions~\ref{prop1}, \ref{prop2} and~\ref{prop3}.

\begin{lem}\label{eqcond6}
Let $P$ be a partial IP~loop and $x,y,z \in P$. Then the following six
conditions are equivalent:
\smallskip

\begin{enum}
\item[\rm{(i)}]\
$(x,y) \in D(P)$ and $xy = z$;
\item[\rm{(ii)}]\
$\bigl(z,y^{-1}\bigr) \in D(P)$ and $zy^{-1} = x$;
\item[\rm{(iii)}]\
$\bigl(x^{-1},z\bigr) \in D(P)$ and $x^{-1}z = y$;
\item[\rm{(iv)}]\
$\bigl(y,z^{-1}\bigr) \in D(P)$ and $yz^{-1} = x^{-1}$;
\item[\rm{(v)}]\
$\bigl(z^{-1},x\bigr) \in D(P)$ and $z^{-1}x = y^{-1}$
\item[\rm{(vi)}]\
$\bigl(y^{-1},x^{-1}\bigr) \in D(P)$ and $y^{-1}x^{-1} = z^{-1}$;.
\end{enum}
\end{lem}

\begin{proof}
Applying the IP~loop property (3) in a proper way and (if necessary) using
the fact that $\bigl(a^{-1}\bigr)^{-1} = a$ for any $a \in P$, we can get the
following cycle of implications:

\smallskip
\centerline{(i)$\imp$(ii)$\imp$(v)$\imp$(vi)$\imp$(iv)$\imp$(iii)$\imp$(i).}

\smallskip
\noindent
We show just the first implication, leaving the remaining ones to the reader.
If $(x,y) \in D(P)$ and $xy = z$ then, according to (3),
$\bigl(z,y^{-1}\bigr) = \bigl(xy,y^{-1}\bigr) \in D(P)$ and
$zy^{-1} = x$.
\end{proof}

The last lemma implies that if any three of the pairs above are gaps in $P$
then so are the remaining three pairs, as well.

In the generic case all the six equivalent conditions above are different.
There are just two kinds of exceptions: first the trivial ones, when at
least one of the elements $x$, $y$, $z$ equals the unit $1$ (which never
produce gaps), and second, if $x = y \in O_3(P)$, when the six conditions
reduce to just two:

\smallskip
\begin{enum}
\item[\bul]\
$(x,x) \in D(P)$ and $xx = x^{-1}$,
\item[\bul]\
$\bigl(x^{-1},x^{-1}\bigr) \in D(P)$ and $x^{-1}x^{-1} = x$.
\end{enum}
\smallskip

From now on we will preferably use a more relaxed language: when writing $xy = z$
for elements $x$, $y$, $z$ of some partial IP~loop $P$ we will automatically
assume that $(x,y) \in D(P)$, without mentioning it explicitly.

The number of gaps in any finite IP~loop $P$ is related to the size of $P$ and that
of the set $O_3(P)$ of order three elements through a congruence modulo~6.

\begin{lem}\label{rad3}
Let $P$ be a finite partial IP~loop. Then
$$
\#\Gama(P) \equiv (\#P - 1)(\#P - 2) - o_3(P) \!\pmod 6.
$$
\end{lem}

\begin{proof}
We know that $(x,1),(1,x),\bigl(x,x^{-1}\bigr) \in D(P)$ for any $x \in P$.
At the same time, $(a,a) \in D(P)$ for all $a \in O_3(P)$. Except for these pairs,
there are other $(\#P - 1)(\#P - 2) - o_3(P)$ pairs which can be either in $D(P)$
or in $\Gama(P)$. Those which are in $D(P)$ can be split into sixtuples according
to Lemma~\ref{eqcond6}, hence their number is divisible by~6, proving the above
congruence.
\end{proof}

We will also use one graph-theoretical result, namely the Dirac's theorem \cite{Dir}
giving a sufficient condition for the existence of a Hamiltonian cycle in a graph.
For our purpose, the term \textit{graph} always means an undirected graph without
loops and multiple edges. For the basic graph-theoretical concepts the reader is
referred to Diestel \cite{Diest}.

\begin{lem}\label{dirac}
Let $G$ be a graph with $n \ge 3$ vertices in which every vertex has the degree at
least $n/2$. Then $G$ has a Hamiltonian cycle.
\end{lem}

\section{Extensions of partial IP~loops and the proof of Proposition~1}\label{3}

\noindent
All the three Propositions~\ref{prop1}, \ref{prop2} and~\ref{prop3} deal with
extensions of a partial IP~loop $(P,\cdot)$, which can be combined using two
more specific types of this construction: first, extensions preserving (the
domain of) the binary operation $\cdot$ on the original partial IP~loop $P$
and extending the base set of $P$, and, second, extensions preserving the base
set of $P$ and extending (the domain of) the binary operation on $P$. We start
with the first type of extensions.

Let $P$, $Q$ be two partial IP~loops such that $P \cap Q = \{1\}$, i.e., their
base sets have just the unit element 1 in common. Then, obviously, the set
$P \cup Q$ can be turned into a partial IP~loop, which we denote by $P \sqcup Q$,
extending both $P$ and $Q$, with domain
$$
D(P \sqcup Q) = D(P) \cup D(Q),
$$
i.e., preserving the original operations on both $P$ and $Q$, and leaving undefined
all the products $xy$, $yx$, for $x \in P \ssm \{1\}$, $y \in Q \ssm \{1\}$.
The partial IP~loop $P \sqcup Q$ is called the \textit{direct sum} of the partial
IP~loops $P$ and $Q$.

Let us fix the notation for some particular cases of this construction, considered
as extensions of the IP~loop $P$ fixed in advance. In all the particular cases below
$A$ denotes a nonempty set disjoint from $P$ such that $(A \cup \{1\},\cdot)$ is
a partial IP~loop.

Let $\sigma\co A \to A$ be an involution, i.e., $\sigma(\sigma(a)) = a$ for any
$a \in A$. Then the \textit{minimal partial IP loop} $[A,\sigma]$ has the base
set $A \cup \{1\}$ and the partial binary operation given by $ 1\cdot 1 = 1$, and
$$
1a = a1 = a, \qquad
a\sigma(a) = \sigma(a)a = 1,
$$
for any $a \in A$, leaving the operation result $ab$ undefined for any other
pair of elements $a,b \in A$. The reader is asked to realize that $[A,\sigma]$ is
indeed a partial IP~loop, and that it is minimal (concerning its domain) among all
partial IP~loops with the base set $A \cup \{1\}$, which satisfy
$$
a^{-1} = \sigma(a)
$$
for any $a \in A$. Then, obviously,
$$
O_2[A,\sigma] = \{a \in A\co \sigma(a) = a\},
$$
i.e., the order two elements in $[A,\sigma]$ coincide with the fixpoints of
the map $\sigma$. The direct sum of the partial IP~loops $P$ and $[A,\sigma]$
is denoted by
$$
P[A,\sigma] = P \sqcup [A,\sigma].
$$
The order two elements in $P[A,\sigma]$ split into two disjoint easily recognizable
parts
$$
O_2(P[A,\sigma]) = O_2(P) \cup O_2[A,\sigma].
$$
If $\sigma = \id_A\co A \to A$ is the identity on $A$, we write
$$
P[A,\id_A] = P[A],
$$
in which case
$$
O_2(P[A]) = O_2(P) \cup A.
$$
If $A = \{a_1, \dots, a_n\}$ is finite, we write
$$
P[A] = P[a_1,\dots,a_n].
$$
In particular, if $A = \{a\}$ is a singleton (and $\sigma = \id_A$ is the unique
map $A \to A$), then
$$
P[\{a\}] = P[a].
$$
If $A = \{a,a'\}$ where $a \ne a'$, and $\sigma$ is the transposition exchanging $a$
and $a'$, we denote
$$
P[A,\sigma] = P[a \lra a'].
$$

From among the second type of extensions of a partial IP~loop $P$, preserving
its base set $P$ and extending just (the domain of) its operation the simplest
ones attempt at filling in just a single gap in $P$. This type of extension will
be called a \textit{simple extension through the relation $xy = z$}. More precisely,
having $x,y,z \in P$ such that $(x,y) \in \Gama(P)$, we want to put $xy = z$.
From Lemma~\ref{eqcond6} it follows that then we have to satisfy the remaining
five relations, too. This is possible only if all the pairs $(x,y)$,
$\bigl(z,y^{-1}\bigr)$, $\bigl(x^{-1},z\bigr)$ (or, equivalently, any other three
pairs occurring there) are gaps in $P$. This is a sufficient condition, as well,
since in that case we can define all the products as required by Lemma~\ref{eqcond6}.
Thus filling in the gap $(x,y)$ enforces to fill in some other related gaps, too.
In that case we automatically assume that the remaining five relations are defined
in accord with Lemma~\ref{eqcond6}.

Iterating simple extensions through particular relations we have have to check
in each step whether any new relation $uv = w$ (and its equivalent forms) does
not interfere not only with the pairs in $D(P)$ but also with the gaps already
filled in by previous simple extensions. In other words, we are interested in
situations when we can fill in a whole set of gaps at once.

If $(P,\cdot)$ is a partial IP~loop and $*$ is a partial operation on the set $P$
with domain $T \sbs P \times P$, such that such that $T \sbs \Gama(P)$ then,
since $D(P) \cap T = \emptyset$, we can extend the original operation $\cdot$ to
the set $D(P) \cup T$ by putting $xy = x*y$ for $(x,y) \in T$. The resulting
structure will be called the \textit{extension of the IP~loop $P$ through the
operation $*$}. The next lemma tells us when such an extension gives us an IP~loop,
again. In its formulation $x^{-1}$ denotes the inverse of the element $x \in P$ with
respect to the original operation $\cdot$ in $P$.

\begin{lem}\label{itersimple}
Let $(P,\cdot)$ be a partial IP~loop and $*$ be a partial binary operation on the
set $P$ with domain $T \sbs \Gama(P)$. Then the extension of the operation $\cdot$
through the operation $*$ to the set $D(P) \cup T$ yields a partial IP~loop extending
$P$ if and only if $T$ and $*$ satisfy the following condition:

\smallskip
\begin{enum}
\item[\rm{(5)}]
for any $x,y,z \in P$, if $(x,y) \in T$ and $x*y = z$ then also
all the pairs $\bigl(z,y^{-1}\bigr)$, $\bigl(x^{-1},z\bigr)$,
$\bigl(y,z^{-1}\bigr)$, $\bigl(z^{-1},x\bigr)$, $\bigl(y^{-1},x^{-1}\bigr)$
belong to $T$ and satisfy all the relations \,$z*y^{-1} = x$, \,$x^{-1}*z = y$,
\,$y*z^{-1} = x^{-1}$, \,$z^{-1}*x = y^{-1}$, \,$y^{-1}*x^{-1} = z^{-1}$.
\end{enum}
\end{lem}

\begin{proof}
In view of Lemma~\ref{eqcond6}, condition (5) obviously is necessary. By the same
reason, condition (5) implies that each of the particular relations $xy = x*y$, for
$(x,y) \in T$, can be separately added to $P$. Since $T \sbs \Gama(P)$, no particular
relation $xy = x*y$ can interfere with the remaining added relations $uv = u*v$.
\end{proof}

The following simple combination of both the types of extensions will be used in
the proof of Proposition~\ref{prop1}.

Let $A$ be a set (disjoint from $P$) and $\sigma\co A \to A$ be a fixpointfree
involution (i.e., $\sigma(a) \ne a$ for every $a \in A$). Then $[A,\sigma]_3$
denotes the extension of the minimal partial IP~loop $[A,\sigma]$ through
(just) the additional relations
$$
aa = \sigma(a)
$$
for any $a \in A$. Formally, $[A,\sigma]_3$ is the extension of $[A,\sigma]$
through the operation $*$ defined on the set $T = \{(a,a)\co a \in A\}$ by
$a*a = \sigma(a)$ for any $a \in A$. It is clear that each pair $(a,a)$ is
indeed a gap in $[A,\sigma]$ and that the condition (5) from Lemma~\ref{itersimple}
is satisfied. Hence $[A,\sigma]_3$ is a partial IP~loop extending $[A,\sigma]$
in which
$$
a^{-1} = \sigma(a) = aa
$$
for each $a \in A$, i.e., every element $a \in A$ has the order three.
For the direct sum
$$
P[A,\sigma]_3 = P \sqcup [A,\sigma]_3
$$
we have
$$
O_3(P[A,\sigma]_3) = O_3(P) \cup A.
$$
If $A = \{a,a'\}$, where $a \ne a'$, then the denotations $[A, a \lra a']_3$
and
$$
P[a \lra a']_3 = P[A, a \lra a']_3
$$
are already self-explanatory, and similarly for
$[A, a \lra a', b \lra b']_3$ and
$$
P[a \lra a', b \lra b']_3 = P[A, a \lra a', b \lra b']_3
$$
where the set $A$ consists of four distinct elements $a$, $a'$, $b$, $b'$.

\begin{proof}[Proof of Proposition 1]
Let $a$, $a'$, $b$, $b'$ be four distinct elements not belonging to $P$.
Let us form the extensions $Q = P[a \lra a']_3$ and
$R = P[a \lra a', b \lra b']_3$. Obviously,
$$
o_3(Q) = o_3(P) + 2 \quad\text{and}\quad o_3(R) = o_3(P) + 4.
$$
Since one of the numbers $o_3(P)$, $o_3(P)+2$, $o_3(P)+4$ is divisible by~3,
one of the partial IP~loops $P$, $Q$, $R$ has the desired property.
\end{proof}

\section{The proof of Proposition~2}\label{4}

\noindent
A more subtle combination of the two types of extensions introduced in
Section~\ref{3} will be required in the proof of Proposition~\ref{prop2}.

\begin{proof}[Proof of Proposition 2]
Let $P$ be a finite partial IP~loop such that $3\,|\,o_3(P)$, and $A$ be a finite
set disjoint from $P$ with the number of its elements satisfying
$$
\#A \ge \max\bigl\{5(\#P) - 1, \#\Gama(P)/2\} \quad \text{and} \quad
10\le \#P + \#A \equiv 4 \!\pmod 6.
$$
First we construct the minimal extension $P[A]$, in which every element $a$ of $A$
has the order two, while
$$
O_3(P[A]) = O_3(P).
$$
Hence the partial IP~loop $P[A]$ has the base set $P \cup A$ with the required
number of elements and the same number of elements of the order three as $P$.

Next, we construct an extension of $P[A]$ in which all the original gaps
in $\Gama(P)$ will be filled. We take $T = \Gama(P) \sbs \Gama(P[A])$ and
introduce a binary operation $*$ on $T$, assigning to each pair of gaps
$(x,y), \bigl(x^{-1},y^{-1}\bigr) \in T$ a (self-inverse) element
$$
x*y = y^{-1}*x^{-1} = (x*y)^{-1}
$$
from $A$. At the same time we arrange that (with the above exception) $x*y \ne u*v$
whenever $(x,y)$ and $(u,v)$ are different gaps in $P$. This is possible, as
$\#A \ge \Gama(P)/2$. Since $(1,a)$, $(a,1)$ and $(a,a)$, where $a \in A$, are
the only gaps in $P[A]$ containing some element of $A$, condition (5) of
Lemma~\ref{itersimple} is obviously satisfied. Thus we can construct the partial
IP~loop $P[A]^*$, extending $P[A]$ through the operation $*$. It still has the base
set $P \cup A$, while
$$
\Gama\bigl(P[A]^*\bigr) \cap (P \times P) = \emptyset.
$$
At the same time, $ab \in P$ for any $(a,b) \in D(P[A]^*) \cap (A \times A)$.

Finally, we construct an extension $Q$ of $P[A]^*$ with the same base set
$P \cup A$, such that
$$
\Gama(Q) \sbs O_2(Q) \times O_2(Q).
$$
As all the elements of $A$ are of order two,
and $P[A]^*$ has no gap $(x,y) \in P \times P$, it suffices to manage that
$(x,a), (a,x) \in D(Q)$ for all $a \in A$, $x \in P \ssm O_2(P)$, $x \ne 1$.

We will proceed by an induction argument. To this end we represent the set
$$
P \ssm \bigl(O_2(P) \cup \{1\}\bigr) =
   \bigl\{x_1, x_1^{-1}, \dots, x_n, x_n^{-1}\bigr\},
$$
in such a way that each pair of mutually inverse elements
$x, x^{-1} \in P \ssm \bigl(O_2(P) \cup \{1\}\bigr)$ occurs in this list exactly once.
To start with we put $Q_0 = P[A]^*$. Now we assume that, for some $0 \le k < n$, we
already have an IP~loop $Q_k \ge P[A]^*$ with the same base set $P \cup A$, satisfying
the following three conditions:

\smallskip
\begin{enum}
\item[(6)]
$au, va \in A$ for any $a \in A$, $u,v \in P \ssm \{1\}$ such that
$(a,u), (v,a) \in D(Q_k)$
\item[(7)]
$ab \in P$ for any $(a,b) \in D(Q_k) \cap (A \times A)$, and
\item[(8)]
$(x_l,a), (a,x_l) \in D(Q_k)$ for all $0 \le l \le k$, $a \in A$,
\end{enum}

\smallskip
\noindent
Observe that $Q_0$ trivially satisfies all these conditions (with $k = 0$), and condition
(8) jointly with Lemma~\ref{eqcond6} imply that
$\bigl(x_l^{-1},a\bigr), \bigl(a,x_l^{-1}\bigr) \in D(Q_k)$ for all $0 \le l \le k$,
$a \in A$, too. For $x = x_{k+1}$, we have to fill in all the gaps in $Q_k$ in which $x$
occurs, preserving all the conditions (6), (7), (8) with $k$ replaced by $k+1$. That way
all the gaps in $Q_k$ containing $x^{-1}$ will be filled in, as well.

Let us introduce the sets
$$
L_x = \{a \in A\co (a,x) \in \Gama(Q_k)\}
\quad \text{and} \quad
R_x = \{a \in A\co (x,a) \in \Gama(Q_k)\}.
$$

\begin{cl}\label{claim1}
We have $\#L_x = \#R_x$.
\end{cl}

\begin{proof}
Since $xu=v$ implies $v^{-1}x = u^{-1}$ for any $u,v \in P \cup A$, we have
a bijection between the sets
\begin{align*}
(P \cup A) \ssm L_x &= \{u \in P \cup A\co (x,u) \in D(Q_k)\}, \\
(P \cup A) \ssm R_x &= \bigl\{v \in P \cup A\co \bigl(v^{-1},x\bigr) \in D(Q_k)\bigr\},
\end{align*}
which implies that the sets $\#L_x$ and $\#R_x$ have the same number of elements.
\end{proof}

Thus there exists a bijective map $\eta\co L_x \to R_x$ (with inverse map
$\eta^{-1}\co R_x \to L_x$); latter on we will specify some additional requirements
concerning it. We intend to use $\eta$ in defining the extending operation $*$ on
the set
$$
T_x = \bigl(L_x \times \bigl\{x,x^{-1}\bigr\}\bigr) \cup
      \bigl(\bigl\{x,x^{-1}\bigr\} \times R_x\bigr) \cup
      \bigl\{(a,\eta(a))\co a \in L_x\bigr\} \cup \bigl\{(\eta(a),a)\co a \in L_x\bigr\}
$$
by putting
$$
a*x = \eta(a)
$$
for any $a \in L_x$. Then we have to satisfy the remaining five conditions of
Lemma~\ref{itersimple}, i.e. (remembering that the elements of $A$ are self-inverse),
$$
a*\eta(a) = x, \quad x^{-1}*a = \eta(a), \quad \eta(a)*a = x^{-1}, \quad
   \eta(a)*x^{-1} = x*\eta(a) = a.
$$
The substitution $b = \eta(a)$ into the last two relations yields
$$
b*x^{-1} = x*b = \eta^{-1}(b)
$$
for any $b \in R_x$. It follows that each pair $(a,\eta(a))$, where $a \in L_x$, must
be a gap in $Q_k$. Since $(a,a) \in D(Q_k)$ for all $a \in A$, this implies that
$\eta(a) \ne a$ for $a \in L_x \cap R_x$ (if any). Additionally, $\eta$ must avoid any
``crossing'', i.e., the situation that
$$
\eta(a) = b \quad\text{and}\quad \eta(b) = a
$$
for some distinct $a,b \in L_x \cap R_x$. This namely, according to Lemma~\ref{eqcond6},
would imply that $a*b = x = b*a$, and, since $(a*b)^{-1} =b*a$, produce a contradiction
$x = x^{-1}$. Now, it is clear, that the partial IP~loop $Q_{k+1}$, to be obtained as the
extension of $Q_k$ through the operation $*$ constructed from the bijection $\eta$ as
described, will satisfy all the conditions (6), (7), (8) (with $k+1$ in place of $k$).
Thus it is enough to show that there is indeed a ``crossing avoiding'' bijection
$\eta\co L_x \to R_x$ such that
$$
(a,\eta(a)) \in \Gama(Q_k)
$$
for each $a \in L_x$. To this end we denote the common value $\#L_x = \#R_x$ by $m$,
enumerate the sets
$$
L_x = \{a_1,\dots,a_m\}, \qquad R_x = \{b_1,\dots,b_m\}
$$
in such a way that $i = j$ whenever $a_i = b_j \in L_x \cap R_x$, and introduce the graph
$G_x$ on the vertex set $V = \{1,\dots,m\}$, joining two vertices $i$, $j$ by an edge if
and only if $i \ne j$ and both $(a_i,b_j), (a_j,b_i) \in \Gama(Q_k)$.

\begin{cl}\label{claim2}
The graph $G_x$ has a Hamiltonian cycle.
\end{cl}

\begin{proof}
According to Lemma~\ref{dirac}, it suffices to show that $m \ge 3$ and that the minimal
degree of vertices in $G_x$ is at least $m/2$. We keep in mind that both the right side
and the left side multiplication in $Q_k$ by a fixed element are injective maps.

Since $ax \in P \ssm \{1\}$ for every $a \in A$ such that $(a,x) \in D(Q_k)$, there are
at most $\#P - 1$ pairs $(a,x)$ in $D(Q_k)$. Hence
$$
m = \#L_x \ge \#A - \#P  + 1 \ge 4(\#P) > 3.
$$
Let $i$ be any vertex in $G_x$. Then $i$ is not adjacent to a vertex $j$ if and only
if at least one of the pairs $(a_i,b_j)$, $(a_j,b_i)$ belongs to $D(Q_k)$. However,
for fixed $a_i$ or $b_i$, all such products $a_ib$ or $ab_i$ belong to $P$ and, in
both cases, every element of $P$ occurs as a result at most once. Thus there are at
most $2(\#P)$ vertices in $G_x$ not adjacent to $i$. Therefore,
$$
\deg(i) \ge m - 2(\#P) \ge m - \frac{m}{2} = \frac{m}{2}.
$$

\vskip-12pt
\end{proof}

Let $\pi$ be a cyclic permutation of the set $V$ such that
$\bigl(1,\pi(1),\dots,\pi^{n-1}(1)\bigr)$ is a Hamiltonian cycle in $G_x$.
We define $\eta\co L_x \to R_x$ by
$$
\eta(a_i) = b_{\pi(i)}
$$
for any $i \in V$. Obviously, $\eta$ is bijective, $(a_i,\eta(a_i)) \in \Gama(Q_k)$
for each $i \in V$, and, since $m \ge 3$, it avoids any crossing.

It follows that in the extension $Q_{k+1}$ of the partial IP~loop $Q_k$ through the
operation $*$ all the gaps from the set $T_x$ are filled in, and the conditions
(6), (7), (8) are preserved. The last partial IP~loop $Q = Q_n$ satisfies already
all the requirements of Proposition~\ref{prop2}.
\end{proof}

\section{Steiner triples and the proof of Proposition~3}\label{5}

\noindent
In the proof of Proposition~\ref{prop3} we will make use of Steiner loops and Steiner
triple systems. A \textit{Steiner loop} is an IP~loop satisfying the identity $xx = 1$,
i.e., an IP~loop in which every element $x \ne 1$ has the order two. Steiner loops are
closely related to \textit{Steiner triple systems}, which are  systems $\St$ of three
element subsets of a given base set $X$ such that each two element subset $\{x,y\}$ of
$X$ is contained in exactly one set $\{x,y,z\} \in \St$. Namely, if $L$ is a Steiner
loop $L$ then $X = L \ssm \{1\}$ becomes a base set of the Steiner triple system
$$
\St_L = \bigl\{\{x,y,xy\}\co x,y \in X\bigr\}.
$$
Conversely, if $\St$ is a Steiner triple system with the base set $X$ then, adjoining
to $X$ a new element $1 \notin X$, we obtain a Steiner loop with the base set
$\X = X \cup \{1\}$, the unit 1 and the operation given by the casework
$$
xy = \begin{cases}
1, & \text{if $x = y$,} \\
z, & \text{where $\{x,y,z\} \in \St$, if $x \ne y$,}
\end{cases}
$$
for $x,y \in X$. Based on this definition, we will call a \textit{Steiner triple} any
three-element set $\{x,y,z\} \sbs O_2(P)$ in any partial IP~loop $P$, such that the
product of any two of its elements equals the third one.

It is well known that there exists a Steiner triple system $\St$ on an $n$-element set
$X$ if and only if $n \equiv 1$ or $n \equiv 3 \!\pmod 6$ (see, e.g., Hwang~\cite{HL}).

The construction reducing eventually the number of gaps in a given partial IP~loop $P$,
satisfying certain conditions which will be emerging gradually, is composed of several s
impler steps, we are going to describe, now. At the same time, it depends on a six term
progression $\ap = (a_0,a_1,a_2,a_3,a_4,a_5)$ of pairwise distinct order two elements of
$P$ chosen in advance; the criteria for its choice will be clarified later on.

The first step is the \textit{triplication construction}, which uses Steiner loops
heavily. Given an arbitrary finite partial IP~loop  $P$ such that
$\#P \equiv 2$ or $\#P \equiv 4$ (mod~6) we denote $n = \#P - 1$. Then Steiner triple
systems on $n$-element sets, as well as Steiner loops on $n+1$-element sets exist;
assume that $Y$, $Z$ are two $n$-element sets, such $P$, $Y$, $Z$ are pairwise disjoint,
and that both the sets $\Y = Y \cup \{1\}$, $\Z = Z \cup \{1\}$ are equipped with binary
operations turning them into Steiner loops. We denote by, in a fairly ambiguous way,
$$
3P = P \sqcup \Y \sqcup \Z
$$
the direct sum of the partial IP~loop $P$ with the Steiner loops $\Y$ and $\Z$ (see
Section~\ref{4}). It is a partial IP~loop with the base set $P \cup Y \cup Z$, consisting
of $3n + 1$ elements, and the domain
$$
D(3P) = D(P) \cup (Y \times Y) \cup (Z \times Z) \cup
       \bigl(\{1\} \times (Y \cup Z)\bigr) \cup \bigl((Y \cup Z) \times \{1\}\bigr).
$$
We will extend the partial operation on $3P$ by filling all the gaps consisting of
pairs of elements of different sets $P$, $Y$, $Z$. That way we'll obtain an extension
$3P^*$ of $3P$ with the same base set $P \cup Y \cup Z$, such that $\Gama(3P^*) = \Gama(P)$.
The extending operation $*$ is defined on the set
$$
T = (P_0 \times (Y \cup Z)) \cup ((Y \cup Z) \times P_0)
     \cup (Y \times Z) \cup (Z \times Y) \sbs \Gama\bigl(P \sqcup \Y \sqcup \Z\bigr),
$$
where $P_0 = P \ssm \{1\}$. It depends on some arbitrary fixed enumerations
$$
P_0 = \{x_0,\dots,x_{n-1}\}, \quad Y = \{y_0,\dots,y_{n-1}\}, \quad
  Z = \{z_0,\dots,z_{n-1}\}
$$
of the sets $P_0$, $Y$, $Z$, respectively. Once having them we put
$$
y_i*z_{i+k} = x_{i+2k}
$$
for $0 \le i, k < n$, with the addition of subscripts modulo $n$. Then, in order to
satisfy the conditions of Lemma~\ref{eqcond6}, we define
\begin{gather*}
x_{i+2k}*z_{i+k} = y_i,\quad y_i*x_{i+2k} = z_{i+k},\quad x_{i+2k}^{-1}*y_i=z_{i+k}, \\
z_{i+k}*x_{i+2k}^{-1}=y_i,\quad z_{i+k}*y_i = x_{i+2k}^{-1},
\end{gather*}
using the fact that all the elements of $Y$ and $Z$ are self-inverse. As all the pairs
$(x,y)$, $(y,x)$, $(x,z)$, $(z,x)$, $(y,z)$, $(z,y)$, where $x \in P_0$, $y \in Y$,
$z \in Z$, are gaps in $3P$, Lemma~\ref{itersimple} guarantees that the extension $3P^*$
of the partial IP~loop $3P$ through the operation $*$ is a partial IP~loop, again. For
lack of better terminology we will call it a \textit{Steiner triplication} of the partial
IP~loop $P$ and suppress the Steiner loops $\Y$, $\Z$ and the particular enumerations in
its notation.

The Steiner triplication $3P^*$ of $P$ satisfies $\Gama(3P^*) = \Gama(P)$, hence it still
has the same number of gaps as $P$. However, Proposition~\ref{prop3} requires us to decrease
this number. This will be achieved in a roundabout way. First we cancel some pairs in the
domain $D(3P^*)$, creating that way the potential to fill in more gaps than we have added.
In order to allow for this next step, $P$ has to satisfy some additional conditions, namely,
$\#P \ge 10$ (i.e., $n \ge 9$) and $o_2(P) \ge 6$. Though the enumerations of the sets $P_0$,
$Y$, $Z$, used in the definition of the extending operation $*$, could have been arbitrary,
we now assume that these sets were enumerated in such a way that the six term progression
\,$\ap = (a_0,a_1,a_2,a_3,a_4,a_5)$ chosen in advance coincides with the sixtuple
$(x_0,x_2,x_1,x_5,x_3,x_{n-3})$ and that $\{y_0,y_1,y_3\}$ is a Steiner triple in $\Y$.
This artificial trick will facilitate us the description of the next step of our construction.

Now, necessarily,
\,$z_0 = y_0x_0 = y_3x_{n-3}$, \,$z_1 = y_0x_2 = y_1x_1$ \,and \,$z_3 = y_1x_5 = y_3x_3$,
\,in other words, we have the following seven Steiner triples in $3P^*$:
\begin{gather*}
\{x_0, y_0, z_0\}, \{x_2, y_0, z_1\}, \{x_1, y_1, z_1\}, \\
\{x_5, y_1, z_3\}, \{x_3, y_3, z_3\}, \{x_{n-3}, y_3, z_0\}, \{y_0, y_1, y_3\}.
\end{gather*}
We delete these triples from the domain of $3P^*$. More precisely, for any one of these
three-element sets we delete from $D(3P^*)$ all the six pairs consisting of its distinct
elements. That way we obtain a partial IP~loop $3P^- \le 3P^*$, called the
\textit{reduction} of $3P^*$ which still is an extension of $P$, however, it has 42
more gaps than $P$ (6 for each Steiner triple).

Instead we introduce some new triples consisting of the same elements, namely
\begin{gather*}
\{x_0, x_2, y_0\}, \{x_2, x_1, z_1\}, \{x_1, x_5, y_1\}, \\
\{x_5, x_3, z_3\}, \{x_3, x_{n-3}, y_3\}, \{x_{n-3}, x_0, z_0\}, \\
\{y_0, y_1, z_1\}, \{y_1, y_3, z_3\}, \{y_3, y_0, z_0\},
\end{gather*}
which are intended to become Steiner triples, after we define a partial operation $\circ$
on the set $\{x_0, x_2, x_1, x_5, x_3, x_{n-3}, y_0, y_1, y_3, z_0, z_1, z_3\}$ by putting
the product of any pair of distinct elements of a given three-element set from this list
equal to the third one. That way we obtain an extending operation of the partial IP~loop
$3P^-$ if and only if all the pairs entering this new operation are gaps in $3P^-$.
This is obviously true for the 18 pairs arising from the three triples in the last row
above. However, this need not be the case for the pairs arising from the six triples
in the first two rows. The problem can be reduced to the question which of the pairs
$(x_0, x_2)$, $(x_2, x_1)$, $(x_1, x_5)$, $(x_5, x_3)$, $(x_3, x_{n-3})$,
$(x_{n-3}, x_0)$ belong to $\Gama(P)$. If, e.g., $(x_0,x_2) \notin \Gama(P)$ then we
cannot put $x_0 \circ x_2 = y_0$, so that $\{x_0, x_2, y_0\}$ cannot become a Steiner
triple.

Therefore, we include just those triples $\{x_i,x_j,y_k\}$ or $\{x_i,x_j,z_k\}$ for
which the pair $(x_i,x_j)$ is a gap in $P$. Every such a ``good'' triple results in
filling in six gaps. We already have 18~gaps filled in thanks to the last row. Thus we
need at least five ``good'' triples in the first two rows in order to fill in additional
30~gaps; this would give $18 + 30 = 48 > 42$ gaps, while having just four ``good''
triples results in refilling back 42~gaps, only. In general, we can fill in
$6(3 + g)$ gaps, where $0 \le g \le 6$ is the number of gaps $(x_i,x_j)$ in the list.

We refer to this last step of the construction as to ``filling in the gaps along the
path'' $\ap$ and denote the final resulting extension of the reduction $3P^-$ by
$3P\la\ap\ra$. Obviously, $3P\la\ap\ra$ is and extension of the original IP~loop $P$,
as well, having by $6(3+g) - 42 = 6(g - 4)$ less gaps than $P$. This number can be negative,
0 or positive, depending on whether $g < 4$, $g = 4$, or $g > 4$. That's why we are
interested just in the case when $g \ge 4$.

After all these preparatory accounts we can finally approach the proof of
Proposition~\ref{prop3}.

\begin{proof}[Proof of Proposition 3]
Let $P$ be a finite partial IP~loop satisfying the conditions
\medskip

\begin{enum}
\item[(4)]\quad
$3\,|\,o_3(P)$, \ $\#P \ge 10$, \ $\#P \equiv 4 \!\pmod 6$ \ and\,
\ $\Gama(P) \sbs O_2(P) \times O_2(P)$,
\end{enum}
\medskip

\noindent
such that $\Gama(P) \ne \emptyset$. We are to show that there is a finite partial
IP~loop $Q \ge P$ satisfying these conditions, as well, with less gaps than $P$.

Since $\Gama(P) \sbs O_2(P) \times O_2(P)$, it is an antireflexive and symmetric
relation on the set $O_2(P)$. Thus we can form the \textit{gap graph} $G(P) = (V,E)$
with the set of vertices
$$
V = \{x \in O_2(P)\co (x,y) \in \Gama(P)\ \text{for some}\ y \in O_2(P)\}
$$
and the set of edges
$$
E = \bigl\{\{x,y\}\co (x,y) \in \Gama(P)\bigr\}.
$$
From the definition of the set of vertices $V$ it follows there are no isolated vertices
in $G(P)$. Let's record some less obvious useful facts about this graph.

\begin{cl}\label{claim3}
\begin{enum}
\item[\rm (a)]
The degree of each vertex in $G(P)$ is even.
\item[\rm (b)]
The number of edges in $G(P)$ is divisible by three.
\end{enum}
\end{cl}

\begin{proof}
(a) Let $x \in O_2(P)$. Then the conditions $xy = z$ and $xz = y$ are equivalent for any
$y,z \in P$. Additionally, as $x \neq 1$, from $xy = z$ it follows that $y \ne z$. Thus
the elements $y \in P$ such that $(x,y)\in D(P)$ can be grouped into pairs, hence their
number is even. As $\#P$ is even, too, so is the degree
$$
\deg(x) = \#\{y \in O_2(P)\co (x,y) \in \Gama(P)\}
        = \#P - \#\{y \in P\co (x,y) \in D(P)\}.
$$

(b) By Lemma~\ref{rad3} we have
$$
\#\Gama(P)\equiv (\#P - 1)(\#P - 2) - o_3(P)\equiv 0 \!\pmod 6.
$$
On the other hand, $\#P\equiv 4\pmod 6$ and $3\,|\,o_3(P)$, yielding $3\,|\,\#\Gama(P)$.
Obviously, the number of edges in $G(P)$ is half of the number of gaps $\#\Gama(P)$,
hence the number of edges in $G(P)$ must be divisible by three.
\end{proof}

The structure of connected components in $G(P)$ obeys the following alternative.

\begin{cl}\label{claim4}
Let $C$ be a connected component of the graph $G(P)$. Then either $C$ contains a triangle
or a path of length five, or, otherwise, $C$ is isomorphic to one of the following graphs:
the cycle $C_4$ of length four, the cycle $C_5$ of length five or the complete bipartite
graph $K_{2,m}$ where $m \ge 4$ is even.
\end{cl}

\begin{proof}
Let $C$ be any connected component in $G(P)$. As $G(P)$ has no isolated vertices and the degree
of every vertex in $C$ is even (and therefore at least two), there is a cycle in $C$. Assume that
$C$ contains no triangle and no path of length five. Then the length of this cycle must be
bigger than three and less than six. Therefore, there are just the following two options:
\smallskip

(a) There is a cycle of length five in $C$. Then there cannot be any other edge coming out
from its vertices since then there would be a path of length five contained in $C$. Thus $C$
coincides with this cycle.
\smallskip

(b) There is a cycle of length four in $C$; let us denote it by $(v_0,v_1,v_2,v_3)$. Then,
as $G(P)$ contains no triangle, neither $\{v_0,v_2\}$ nor $\{v_1,v_3\}$ is an edge in $G(P)$.
If there are no more vertices in $C$ then $C$ is a cycle of length four.

Otherwise, we can assume, without loss of generality, that there is a fifth vertex $u_0 \in C$
adjacent to $v_0$. As $u_0$ has an even degree, it must be adjacent to some other vertex,
too. If it were adjacent to some vertex $u_1$, distinct from all the vertices $v_0$, $v_1$,
$v_2$, $v_3$, there would be a path $(u_1,u_0,v_0,v_1,v_2,v_3)$ of length five in $C$. If $u_0$
were adjacent to $v_1$ or to $v_3$, there would be a triangle $(u_0,v_0,v_1)$ or $(u_0,v_0,v_3)$
in $C$. That means that $\{u_0,v_2\}$ must be an edge in $G(P)$ and $\deg(u_0) = 2$.

It follows that every other vertex in $C$ must have the degree two and it must be adjacent either
to $v_0$ and $v_2$ or to $v_1$ and $v_3$. However, the second option is impossible, since in that
case $(u_0,v_0,v_1,u_1,v_3,v_2)$ would be a path of length five. This means that $C$ is isomorphic
to the complete bipartite graph $K_{2,m}$, where one term of this partition is formed by the set
$\{v_0,v_2\}$ and the second one by the rest of the vertices in $C$. Since every vertex has an even
degree, $m$ must be even. At the same time, $m \ge 4$, as $K_{2,2}$ has just four vertices (and it
is isomorphic to the cycle $C_4$).
\end{proof}

Thus the proof of Proposition~\ref{prop3} will be complete once we show how to construct
the extension $Q$ in each of the cases listed in Claim~\ref{claim4}.
\smallskip

(a) \textit{$G(P)$ contains a triangle,} i.e., a three-element set of vertices $\{x,y,z\}$
such that all its two-element subsets are edges. Then we can extend $P$ through the operation
$*$ turning $\{x,y,z\}$ into a Steiner triple. The corresponding extension $Q$ of $P$ has all
the properties required and by six less gaps than $P$.
\smallskip

(b) \textit{$G(P)$ contains a path $\ap = (a_0,a_1,a_2,a_3,a_4,a_5)$ of length five.}
Then we can form the Steiner triplication $3P^*$ of $P$ and, filling in the gaps along
the path $\ap$ in its reduction $3P^-$, we obtain the final extension $Q = 3P\la\ap\ra$
satisfying the condition (4), again. If $(a_5,a_0)$ is a gap in $P$ (i.e., if $\ap$ is
a cycle of length five in $G(P)$) then $Q$ has twelve gaps less than $P$, otherwise it
still has by six gaps less than $P$.
\smallskip

We still have to prove Proposition~\ref{prop3} in the case there is neither any triangle nor
any path of length five in $G(P)$. To this end it is enough to construct, in each of the
remaining cases listed in Claim~\ref{claim4}, an extension $Q$ of $P$ such that the graph
$G(Q)$ has the same number of edges as $G(P)$, however, there is a path $\bp$ of length five
in $G(Q)$. From such a $Q$ we can construct another extension $3Q\la\bp\ra \ge Q \ge P$ with
a smaller number of gaps and still satisfying the condition (4), similarly as we did in the
case (b). So let us have a closer look at the remaining cases.
\smallskip

(c) \textit{$G(P)$ contains a connected component isomorphic to $K_{2,m}$, where $m \ge 4$.}
Let $\{u_0,u_1\}$ be the two-element term of the partition and $\{v_0,v_1,v_2,v_3\}$ be any
four-element subset of the second partition term. We denote by $\ap$ the six-term progression
$(v_0,u_0,v_1,v_2,u_2,v_3)$ and construct the extension $Q = 3P\la\ap\ra$ of the partial
IP~loop $P$ with the gap graph $G(Q)$. Then $\{v_0,u_0\}$, $\{u_0,v_1\}$, $\{v_2,u_2\}$,
and $\{u_2,v_3\}$ are edges in $G(P)$, while $\{v_1,v_2\}$ and $\{v_3,v_0\}$ are not. Hence
the new graph $G(Q)$ has the same number of edges as $G(P)$ and $Q$ has the same number of
gaps as $P$. At the same time, there are two distinct new vertices $y_1$, $z_3$ in $G(Q)$,
occurring in the enumerations of the sets $Y$, $Z$, respectively. Now, one can easily verify
that $\bp = (v_0,u_1,v_1,y_1,v_2,u_0)$ is a path of length five in $G(Q)$.
\smallskip

(d) \textit{There are two distinct connected components $C$ and $D$ in $G(P)$, each of
them isomorphic to the cycle $C_4$ or $C_5$.} Let $m$ and $l$ denote any of the numbers $4$
or $5$. We assume that $(u_0,u_1,\dots,u_{m-1})$ and $(v_0,v_1,\dots,v_{l-1})$ are the
cycles forming the components $C \cong C_m$ and $D \cong C_l$, respectively. Now we take the
six term progression $\ap = (u_0,u_1,u_2,v_0,v_1,v_2)$ and form the extension $Q = 3P\la\ap\ra$.
Once again, $\{u_0,u_1\}$, $\{u_1,u_2\}$, $\{v_0,v_1\}$ and $\{v_1,v_2\}$ are edges in $G(P)$,
while $\{u_2,v_0\}$ and $\{v_2,u_0\}$ are not. Hence $G(Q)$ has the same number of edges as
$G(P)$ and $Q$ has the same number of gaps as $P$. Now, picking the new distinct vertices
$y_1 \in Y$, $z_3 \in Z$, we obtain the path $\bp = (u_3,u_2,y_1,v_0,v_{l-1},v_{l-2})$ of
length five in $G(Q)$.
\smallskip

(e) \textit{$G(P)$ consists of a single connected component isomorphic either to $C_4$ or to $C_5$.}
However, this is impossible, since number of edges in $G(P)$ is divisible by three.
\smallskip

This concludes the proof of Proposition~\ref{prop3}, as well as of Theorems~\ref{thm1}
and~\ref{thm2}.

\end{proof}

\section{Final remarks}\label{6}

\noindent
The discussion from the introduction together with Theorem~\ref{thm2} naturally
lead to the following question.

\begin{prob}\label{prob1}
Is there some minimal (ore even the least) axiomatic class $\mathbf{K}$ of IP~loops
such that every group is locally embeddable into $\mathbf{K}_{\mathrm{fin}}$? Does
this class (if it exists) satisfy the Finite Embedding Property?
\end{prob}

The first candidate which should be examined in this connection seems to be the class
of all Moufang loops. One possible way how to define this concept reads as follows
(see Pflugfelder~\cite{Pflug}, Gagola~\cite{Gag}): A~\textit{Moufang loop} is a loop
satisfying the identity
$$
x(y(xz)) = ((xy)x)z
$$
It is well known that every Moufang loop is an IP~loop.

The following is not the usual definition of the concept of a sofic group
(see Ceccherini-Silberstein, Coornaert \cite{CSC}), however, as proved by
Gordon and Glebsky \cite{GG2}, it is equivalent to it. A group $(G,\cdot)$
is \textit{sofic} if for every finite set $X \sbs G$ and every $\eps > 0$
there exists a finite quasigroup $(Q,*)$ such that $X \sbs Q$,
$$
\frac{\#\{q \in Q\co 1 * q \ne q\}}{\# Q} < \eps,
$$
where $1$ denotes the unit element in $G$, and, for any $x,y \in X$ such that
$x \cdot y \in X$, we have $x \cdot y = x * y$, as well as
$$
\frac{\#\{q \in Q\co (x * y) * q \ne x * (y * q)\}}{\# Q} < \eps.
$$

Theorem~\ref{thm2} together with the above description of sofic groups indicate
that the sofic groups could perhaps be characterized as groups locally embeddable
into some ``nice'' subclass of the class of finite IP~loops, fulfilling some
``reasonable amount of associativity''. A natural candidate is the class of all
finite Moufang loops, once again. For some additional reasons in favor of this choice
see \cite{Gag}.

As already indicated, one should start with trying to clarify the following question.

\begin{prob}\label{prob2}
Does the class of all Moufang loops have the FEP?
\end{prob}

If the answer is negative then it would make sense to elaborate on the following
problem.

\begin{prob}\label{prob3}
Characterize those groups which are locally embeddable into finite Moufang loops.
\end{prob}

Finally, let us formulate two possible responses to Problem~\ref{prob3}, adding as
a comment that we find the first of them (which would follow from the affirmative
answer to Problem~\ref{prob2}) more probable to be true than the second one.

\begin{conj}\label{conj1}
Every group is locally embeddable into finite Moufang loops.
\end{conj}

\begin{conj}\label{conj2}
A group $G$ is sofic if and only if it is locally embeddable into finite Moufang loops.
\end{conj}

\noindent
\textbf{Acknowledgement.}
The second author acknowledges with thanks the support by the grant
no.~1/0333/17 of the Slovak grant agency VEGA.

\end{document}